\documentclass[12pt,a4paper]{amsart}

\usepackage{graphicx,amssymb}
\usepackage[top=40mm,bottom=35mm,left=30mm,right=30mm]{geometry}
\usepackage[frame]{xy}
\xyoption{arc}
\xyoption{graph}
\xyoption{matrix}

\theoremstyle{plan}
\newtheorem{theorem}{Theorem}[section]
\newtheorem{lemma}[theorem]{Lemma}
\newtheorem{corollary}[theorem]{Corollary}
\newtheorem{proposition}[theorem]{Proposition}

\theoremstyle{definition}
\newtheorem{definition}[theorem]{Definition}
\newtheorem{question}[theorem]{Question}
\newtheorem{conjecture}[theorem]{Conjecture}

\let\le\leqslant
\let\ge\geqslant
\def\supp{\operatorname{supp}}
\def\Lk{{\operatorname{Lk}}}

\begin{document}

\title[Embeddability of RAAGs on complements of trees]
   {Embeddability of right-angled Artin groups\\ on complements of trees}

\author{Eon-Kyung Lee and Sang-Jin Lee}
\address{Department of Mathematics, Sejong University, Seoul, Korea}
\email{eonkyung@sejong.ac.kr}

\address{Department of Mathematics, Konkuk University, Seoul, Korea;
School of Mathematics, Korea Institute for Advanced Study, Seoul, Korea}
\email{sangjin@konkuk.ac.kr}
\date{\today}

\begin{abstract}
For a finite simplicial graph $\Gamma$, let $A(\Gamma)$ denote the right-angled
Artin group on $\Gamma$.
Recently Kim and Koberda introduced the extension graph $\Gamma^e$
for $\Gamma$, and established the Extension Graph Theorem:
for finite simplicial graphs $\Gamma_1$ and $\Gamma_2$
if $\Gamma_1$ embeds into $\Gamma_2^e$ as an induced subgraph
then $A(\Gamma_1)$ embeds into $A(\Gamma_2)$.
In this article we show that the converse of this theorem
does not hold for the case $\Gamma_1$ is the complement of a tree
and for the case  $\Gamma_2$ is the complement of a path graph.

\medskip\noindent
{\em Keywords\/}:
right-angled Artin groups, extension graphs, embeddability.\\
{\em 2010 Mathematics Subject Classification\/}: Primary 20F65; Secondary 05C25
\end{abstract}

\maketitle

\section{Introduction}\label{sec:intro}

Throughout this article all graphs are assumed to be simplicial and undirected.

Let $\Gamma$ be a finite   graph with vertex set $V(\Gamma)$ and edge set $E(\Gamma)$.

For a subset $A\subset V(\Gamma)$, the subgraph $\Gamma_1$ of $\Gamma$ with $V(\Gamma_1)=A$ and
$E(\Gamma_1)=\{\{a,b\}\in E(\Gamma) :a,b\in A\}$
is called the \emph{subgraph of\/ $\Gamma$ induced by $A$}
or the \emph{induced subgraph of\/ $\Gamma$ on $A$}.
If a graph $\Gamma_1$ embeds into $\Gamma$ as an induced subgraph,
we write $\Gamma_1\le\Gamma$.
The \emph{complement graph} of $\Gamma$, denoted by $\bar\Gamma$,
is the graph such that $V(\bar\Gamma)=V(\Gamma)$ and two vertices
are adjacent in $\bar\Gamma$ if and only if they are not adjacent in $\Gamma$.
If a group $H$ embeds into a group $G$, we write $H\le G$.
For elements $g,h$ of a group, let $g^h$ and $[g,h]$ denote the conjugate $h^{-1}gh$
and the commutator $g^{-1}h^{-1}gh$, respectively.

The \emph{right-angled Artin group} $A(\Gamma)$ on $\Gamma$
is defined by the presentation
$$ A(\Gamma)=\langle\, v\in V(\Gamma)\mid [a,b]=1\
\mbox{if $\{a,b\}\in E(\Gamma)$}\,\rangle.
$$
It is well-known that two right-angled Artin groups
$A(\Gamma_1)$ and $A(\Gamma_2)$ are isomorphic as groups
if and only if $\Gamma_1$ and $\Gamma_2$ are isomorphic
as graphs~\cite{Dro87} and that $\Gamma_1\le\Gamma_2$ implies $A(\Gamma_1)\le A(\Gamma_2)$.

\subsection{Embeddability between right-angled Artin groups}

The following is a fundamental question for right-angled Artin groups~\cite{CSS08,KK13}.

\begin{question}[Embeddability Problem]\label{qn:embed}
Is there an algorithm to decide whether or not there exists an embedding
between two given right-angled Artin groups?
\end{question}

Kim and Koberda~\cite{KK13} introduced the notion of extension graph $\Gamma^e$
which is obtained from $\Gamma$ through a combinatorial procedure,
and developed the Extension Graph Theorem.

\begin{definition}[Extension graph]
For a finite   graph $\Gamma$, the \emph{extension graph} of $\Gamma$
is the graph $\Gamma^e$ such that
the vertices of $\Gamma^e$ are in one-to-one correspondence with the conjugates
of vertices of $\Gamma$ in $A(\Gamma)$, that is,
$$V(\Gamma^e)=\{\, a^g\in A(\Gamma): a\in V(\Gamma),\ g\in A(\Gamma)\,\}
$$
and two vertices of $\Gamma^e$ are adjacent
if and only if they commute in $A(\Gamma)$, that is,
$$
E(\Gamma^e)=\{\,\{a^g,b^h\}: a^g,b^h\in V(\Gamma^e),\
\mbox{$[a^g,b^h]=1$ in $A(\Gamma)$}\,\}.
$$
\end{definition}

Extension graphs are usually infinite and locally infinite.

\begin{theorem}[Extension Graph Theorem~\cite{KK13}]\label{thm:EGT}
For finite   graphs $\Gamma_1$ and $\Gamma_2$,
if\/ $\Gamma_1\le\Gamma_2^e$ then $A(\Gamma_1)\le A(\Gamma_2)$.
\end{theorem}

This theorem was a significant progress toward solving the Embeddability Problem.
Recently, Casals-Ruiz showed the following.

\begin{theorem}[Theorem 3.5 in \cite{Cas15}]\label{thm:EGE}
There exists an algorithm that given two finite   graphs $\Gamma_1$ and $\Gamma_2$
decides whether or not $\Gamma_1$ embeds into $\Gamma_2^e$.
\end{theorem}

Due to Theorems~\ref{thm:EGT} and~\ref{thm:EGE},
the Embeddability Problem is solvable for the class of right-angled Artin groups
for which the converse of the Extension Graph Theorem holds.
It is natural to ask for which graphs this converse holds.

\begin{question}[Question 1.5 in~\cite{KK13}]\label{qn:ECG}
For which graphs $\Gamma_1$ and $\Gamma_2$,  do we have
$A(\Gamma_1)\le A(\Gamma_2)$
if and only if $\Gamma_1\le\Gamma_2^e\,$?
\end{question}

If a graph $\Lambda$ does not embed into a graph $\Gamma$ as an induced subgraph,
we say that $\Gamma$ is \emph{$\Lambda$-free}.
Let $P_n$ and $C_n$ denote the path graph and the cycle on $n$ vertices, respectively.

It is known that the converse of the Extension Graph Theorem holds,
hence $A(\Gamma_1)\le A(\Gamma_2)$ if and only if $\Gamma_1\le\Gamma_2^e$,
for some important classes of graphs:

\begin{enumerate}
\item[(i)] $\Gamma_1$ is a forest
\cite[Corollary 1.9]{KK13};

\item[(ii)] $\Gamma_2$ is $C_3$-free
\cite[Theorem 1.11]{KK13};

\item[(iii)]
$\Gamma_2$ is $C_4$-free and $P_4$-free
\cite[Theorem 5.1]{CDK13}.
\end{enumerate}

The converse of the Extension Graph Theorem does not always hold.
Casals-Ruiz, Duncun and Kazachkov first gave an example~\cite{CDK13}.

\subsection{Main result}

The right-angled Artin groups on complements of trees
form an important class of right-angled Artin groups
because any right-angled Artin group embeds into a right-angled Artin
group on the complement of a tree~\cite{KK15,LL16}.

In this article we show that the converse of the Extension Graph Theorem
does not hold for the case $\Gamma_1$ is the complement of a tree
and for the case  $\Gamma_2$ is the complement of a path graph.

\medskip\noindent\textbf{Main Theorem} (Corollary~\ref{colo:main}).\/
\emph{There exist a finite tree $T$ and a finite path graph $P$ such that
$A(\bar T)$ embeds into $A(\bar P)$ but $\bar T$ does not embed into
$\bar P^e$ as an induced subgraph.}

\subsection{Opposite convention}
From now on, we adopt the opposite of the usual convention
for right-angled Artin groups
as it is more appropriate for our approach:
$$ G(\Gamma)=\langle\, v\in V(\Gamma)\mid [v_i,v_j]=1\
\mbox{if $\{v_i,v_j\}\not\in E(\Gamma)$}\,\rangle.
$$
Namely, $G(\Gamma)=A(\bar\Gamma)$.
For the extension graph we write
$\Gamma^E=\overline{\bar \Gamma^e}$.
The vertices of $\Gamma^E$ coincide with the vertices of $\bar\Gamma^e$
and two vertices of $\Gamma^E$ are adjacent if and only if
they do not commute in $G(\Gamma)$, that is,
\begin{align*}
V(\Gamma^E)&=\{\, a^g\in G(\Gamma): a\in V(\Gamma),\ g\in G(\Gamma)\,\},\\
E(\Gamma^E)&=\{\,\{a^g,b^h\}: a^g,b^h\in V(\Gamma^E),\
\mbox{$[a^g,b^h]\ne 1$ in $G(\Gamma)$}\,\}.
\end{align*}

Note that for   graphs $\Gamma_1$ and $\Gamma_2$ the following hold:
\begin{itemize}
\item[(i)] $\Gamma_1\le\Gamma_2$ if and only if $\bar\Gamma_1\le\bar\Gamma_2$;
\item[(ii)] $\Gamma_1\le\Gamma_2^E$ if and only if $\bar\Gamma_1\le\bar\Gamma_2^e$.
\end{itemize}
Therefore the Extension Graph Theorem is equivalent to
``For finite   graphs $\Gamma_1$ and $\Gamma_2$,
if\/ $\Gamma_1\le\Gamma_2^E$ then $G(\Gamma_1)\le G(\Gamma_2)$.''

\subsection{Our approach toward proving Main Theorem}

Let $T_{p,q,r}$ be the tripod whose three leaves
contain $p$, $q$ and $r$ vertices, respectively.
For instance, $T_{3,2,2}$ and $T_{2,2,2}$ are illustrated in Figure~\ref{fig:Tripods}.
Let $T_2$ denote the tripod $T_{2,2,2}$.

\begin{figure}
\begin{tabular}{*5c}
$\begin{xy}
(-20,0) *{\bullet}; (-10,0) *{\bullet} **@{-}; (0,0) *{\bullet} **@{-}; (10,0) *{\bullet} **@{-};
(17,7) *{\bullet} **@{-}; (24,14) *{\bullet} **@{-};
(10,0); (17,-7) *{\bullet} **@{-}; (24,-14) *{\bullet} **@{-};
\end{xy}$
&\qquad\qquad\qquad&
$\begin{xy}
(-10,0) *{\bullet}; (0,0) *{\bullet} **@{-}; (10,0) *{\bullet} **@{-};
(17,7) *{\bullet} **@{-}; (24,14) *{\bullet} **@{-};
(10,0); (17,-7) *{\bullet} **@{-}; (24,-14) *{\bullet} **@{-};
\end{xy}$\\[3mm]
(a) $T_{3,2,2}$ &&
(b) $T_{2,2,2}=T_2$
\end{tabular}
\caption{Tripods}
\label{fig:Tripods}
\end{figure}

We obtain the Main Theorem by proving the following:
\begin{itemize}
\item[(i)] $G(T_2)$ embeds into $G(P_{22})$ (Theorem~\ref{thm:A});
\item[(ii)] $T_2$ does not embed into $P_n^E$ as an induced subgraph for any $n$ (Theorem~\ref{thm:B}).
\end{itemize}

The non-embeddability of $T_2$ into $P_n^E$
gives rise to a question: Which trees $T$ admit
an embedding into $P_n^E$ as an induced subgraph for some $n$?
For this class of trees $T$,
the right-angled Artin group $G(T)=A(\bar T)$ embeds into $G(P_n)=A(\bar P_n)$ for some $n$.
In Theorem~\ref{thm:hairy}, we obtain a characterization of such trees.

\subsection{Universal family of graphs for right-angled Artin groups}
\label{sec:universal}

Let us say that a collection $\mathcal G$ of finite   graphs is
a \emph{universal family of graphs for right-angled Artin groups}
if for any right-angled Artin group $G(\Gamma_1)$
there exists $\Gamma_2\in\mathcal G$ such that $G(\Gamma_1)\le G(\Gamma_2)$.
Kim and Koberda showed that the family of finite trees provides
a universal family of graphs for right-angled Artin groups.

\begin{theorem}[\cite{KK15,LL16}]\label{thm:antitree}
For any finite   graph $\Gamma$,
there exists a finite tree $T$ such that $G(\Gamma)\le G(T)$.
\end{theorem}

Using the fact that if $\Gamma_1$ is an edge-contraction of $\Gamma_2$
then $G(\Gamma_1)\le G(\Gamma_2)$~\cite{Kim08,KK13}, we can see
that the family of finite trees with degree $\le 3$ at each vertex
is also universal~\cite{Kat16}.
It would be interesting to find a universal family
smaller than this. For instance, one can ask whether or not
the family of path graphs is universal.

\begin{question}
Which right-angled Artin group embeds into $G(P_n)$ for some $n$?
\end{question}

Concerning the above question, Katayama proposed the following question~\cite{Kat16}.

\begin{question}[Question 5.2 in \cite{Kat16}]
Is it possible that $G(T_2)\le G(P_n)$ for some $n$?
\end{question}

Theorem~\ref{thm:A} gives an affirmative answer to the above question,
and Theorem~\ref{thm:hairy} gives a family of trees $T$ with $G(T)\le G(P_n)$.

We ask the same question as above for $T_{p,q,r}$:
For which $p,q,r$, do we have $G(T_{p,q,r})\le G(P_n)$ for some $n$?
It seems very hard to embed $G(T_{p,q,r})$ into $G(P_n)$ if $p,q,r$ are large.
We therefore propose the following:

\begin{conjecture}
If $p,q,r$ are large enough,
then $G(T_{p,q,r})$ does not embed into $G(P_n)$ for any $n$.
\end{conjecture}

\subsection{A remark on Theorem 3.14 in~\cite{Cas15}}
Theorem 3.14 of~\cite{Cas15} claims the following:
``For a forest $F$ and a finite   graph $\Gamma$,
$G(F)\le G(\Gamma)$ if and only if
$F\le\Gamma^E$''.
In other words, it claims that the converse of the Extension Graph Theorem holds
for complements of forests.
This conflicts to our Main Theorem.

In the proof of Theorem 3.14 in~\cite{Cas15}, the following argument is used.
``For $1\le i\le k$, let $g_i\in G(\Gamma)$ be a product of mutually commuting elements,
i.e.\ $g_i=y_{i,1}y_{i,2}\cdots y_{i,r_i}$ such that $[y_{i,p},y_{i,q}]=1$ for all $1\le p<q\le r_i$.
Using commutator identities, the iterated commutator
$[g_1,g_2,g_3,\ldots,g_k]=[\ldots[[g_1,g_2],g_3],\ldots,g_k]$
can be written as a product
$$\prod_{s=(j_1,\ldots,j_k),\ 1\le j_i\le r_i} [y_{1,j_1},y_{2,j_2},\ldots,y_{k,j_k}]^{g(s)}
$$
for some elements $g(s)\in G(\Gamma)$.''
This is however not the case.
For instance, let us denote the vertices of the path graph $P_5$
by $x_1,\ldots,x_5$ as follows.
$$
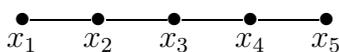
\begin{xy}
(-20,2) *{\bullet};
(-10,2) *{\bullet} **@{-};
(  0,2) *{\bullet} **@{-};
( 10,2) *{\bullet} **@{-};
( 20,2) *{\bullet} **@{-};
(-20,-1) *{x_1};
(-10,-1) *{x_2};
(  0,-1) *{x_3};
( 10,-1) *{x_4};
( 20,-1) *{x_5};
\end{xy}
$$
In $G(P_5)$, $[x_i,x_j]=1$ if and only if $|i-j|\ge 2$.
A direct computation shows
$$[x_2x_4,x_3,x_1,x_5]\ne 1.$$
Notice that
$[x_2,x_3,x_1,x_5]=[x_4,x_3,x_1,x_5]=1$.
If the argument in~\cite{Cas15} were correct, then
$$[x_2x_4,x_3,x_1,x_5]=[x_2,x_3,x_1,x_5]^g [x_4,x_3,x_1,x_5]^h$$
for some $g,h\in G(P_5)$, which results in $[x_2x_4,x_3,x_1,x_5]=1$.
It is a contradiction.

\subsection{Organization}
Section 2 reviews basic materials.
Section 3 shows that $G(T_2)$ embeds into $G(P_{22})$.
Section 4 shows that $T_2$ does not embed into $P_n^E$ as an induced subgraph
for any $n$.

\section{Preliminaries}

Let $\Gamma$ be a finite   graph.
For a vertex $v\in V(\Gamma)$, the \emph{link} of $v$ in $\Gamma$ is the set
$\Lk_\Gamma(v)=\{\, u\in V(\Gamma)\mid \{v,u\}\in E(\Gamma) \,\}$.
We simply write $\Lk(v)$ for $\Lk_\Gamma(v)$ if $\Gamma$ is clear from context.
For $A \subset V(\Gamma)$, we denote by $\Gamma{\backslash} A$
the subgraph of $\Gamma$ induced by $V(\Gamma){\backslash} A$.

Each element in $V(\Gamma)\cup V(\Gamma)^{-1}$ is called a \emph{letter}.
An element in $G(\Gamma)$ can be expressed as a word, which is
a finite product of letters.
Abusing notation, we shall sometimes regard a word as the group element represented by that word.
Let $w=a_1\cdots a_k$ be a word in $G(\Gamma)$
where $a_1,\ldots,a_k$ are letters.
We say that $w$ is \emph{reduced} if any other word representing the same element
in $G(\Gamma)$ as $w$ has at least $k$ letters.

For $g\in G(\Gamma)$, the \emph{support} of $g$, denoted by $\supp(g)$,
is the set of vertices $v$ such that $v$ or $v^{-1}$ appears
in a reduced word for $g$.
It is known that $\supp(g)$ is well-defined.

Let $w$ be a (possibly non-reduced) word in  $G(\Gamma)$.
A subword $v^{\pm1}w_1v^{\mp1}$ of $w$ is called a \emph{cancellation} of $v$ in $w$
if ${\operatorname{supp}}(w_1)\cap {\operatorname{Lk}}_\Gamma(v)=\emptyset$.
If, furthermore, no letter in $w_1$ is equal to $v$ or $v^{-1}$,
it is called an \emph{innermost cancellation} of $v$ in $w$.
It is known that $w$ is reduced if and only if $w$ has no innermost cancellation.

Let $\phi:\Gamma_2\to\Gamma_1$ be a graph homomorphism between finite graphs,
i.e.\ a function
from $V(\Gamma_2)$ to $V(\Gamma_1)$ that maps adjacent vertices to adjacent vertices.
Then $\phi$ induces a group homomorphism $\phi^*:G(\Gamma_1)\to G(\Gamma_2)$
defined by
$$\phi^*(v)=\prod\limits_{v'\in \phi^{-1}(v)} v'$$
for $v\in V(\Gamma_1)$, where the product is defined to be the identity
if $\phi^{-1}(v)$ is the empty set.
Since $\phi$ is a graph homomorphism and since $\Gamma_1$ has no loops,
no two vertices of $\phi^{-1}(v)$ are adjacent, hence the product is well-defined.
Abusing notation, for a word $w$ in $G(\Gamma_1)$, $\phi^*(w)$ denotes
the word defined by the product.
For this, we may fix a total order on $V(\Gamma_2)$
and write each product $\prod_{v'\in \phi^{-1}(v)} v'$
in the increasing order.

\begin{definition}
We say that $\phi:\Gamma_2\to\Gamma_1$ is \emph{$v'$-surviving} for $v'\in V(\Gamma_2)$
if for any reduced word $w$ in $G(\Gamma_1)$,
the word $\phi^*(w)$ has no innermost cancellation of $v'$.
\end{definition}

The following lemma follows from well-known solutions to the word problem in
right-angled Artin groups~\cite{Cha07}.

\begin{lemma} \label{lem:comm1}
Let $a\in V(\Gamma)$ and $w\in G(\Gamma)$.
Then $[a,w]=1$ if and only if\/
$\Lk(a)\cap\supp(w)=\emptyset$, i.e.\ $[a,c]=1$ for each $c\in\supp(w)$.
\end{lemma}

\begin{lemma}\label{lem:comm2}
Let $w^{-1}bw$ be a reduced word for $b\in V(\Gamma)$ and $w\in G(\Gamma)$. Then
\begin{itemize}
\item[(i)] $\supp(w^{-1}bw)=\{b\}\cup\supp(w)$;
\item[(ii)] $\supp(w^{-1}bw)$ spans a connected subgraph of $\Gamma$.
\end{itemize}
\end{lemma}

\begin{proof}
(i)\ \ Notice that if $w_1w_2$ is a reduced word, then $\supp(w_1w_2)=\supp(w_1)\cup\supp(w_2)$
and that $\supp(w^{-1})=\supp(w)$ for any word $w$.
Since $w^{-1}bw$ is reduced, one has
$\supp(w^{-1}bw)=\supp(w^{-1})\cup \{b\}\cup\supp(w)=\{b\}\cup\supp(w)$.

(ii)\ \
Let $\Gamma_0$ be the subgraph of $\Gamma$ induced by $\supp(w^{-1}bw)$.
Assume that $\Gamma_0$ is not connected.
Let $\Gamma_1$ be the component of $\Gamma_0$ containing $b$,
and let $\Gamma_2=\Gamma_0\setminus\Gamma_1$.
Then $w=w_2w_1$ for some nontrivial reduced words $w_1\in G(\Gamma_1)$ and $w_2\in G(\Gamma_2)$
because $[a_1,a_2]=1$ for $a_1\in V(\Gamma_1)$ and $a_2\in V(\Gamma_2)$.
Since each vertex of $\Gamma_2$ commutes with $b$,
we have $w^{-1}bw=w_1^{-1}w_2^{-1}bw_2w_1=w_1^{-1}bw_1$.
This contradicts the hypothesis that $w^{-1}bw$ is reduced.
\end{proof}

\section{Two local moves on graphs}

In this section,
we propose two local moves on graphs that give rise to
an embedding between right-angled Artin groups.
Combining with a result in~\cite{LL16}, we obtain $G(T_2)\le G(P_{22})$.

\begin{proposition}\label{prop:deg1k}
Let $\Gamma_1$ be a finite   graph with a degree $k+2$ vertex $x$
for $k\ge 1$.
Let $\Lk(x)=\{a_1,\ldots,a_k\}\cup\{b,c\}$.
Suppose each $a_i$ has degree 1.
See Figure~~\ref{fig:deg1k}(a).
Let $\Gamma_2$ be the graph obtained from $\Gamma_1$ by deleting
$a_1,\ldots,a_k$ and then by replacing $x$ with the path graph $P_{2k+1}$
as in Figure~\ref{fig:deg1k}(b).
Then $\Gamma_1\le\Gamma_2^E$ and hence $G(\Gamma_1)\le G(\Gamma_2)$.
\end{proposition}

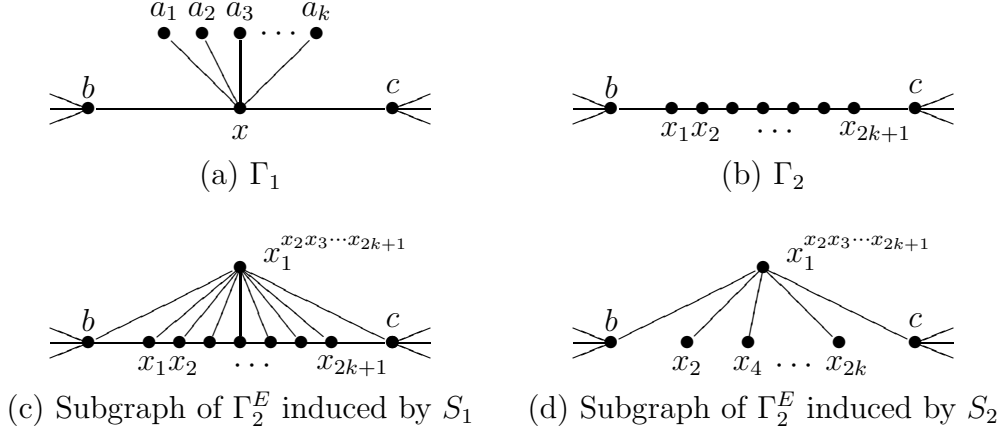
\begin{figure}
\begin{tabular}{*5c}
$\begin{xy}
(-20,0) *{\bullet};  (20,0) *{\bullet} **@{-}; (0,0)*{\bullet};
(-20, 3) *{b};
( 20, 3) *{c};
(0,-3) *{x};
(0,0); (-10,10) *{\bullet} **@{-};
(0,0); (-5,10) *{\bullet} **@{-};
(0,0); (-0,10) *{\bullet} **@{-};
(0,0); (10,10) *{\bullet} **@{-};
(-10, 13) *{a_1};
(-5, 13) *{a_2};
( 0, 13) *{a_3};
( 5, 10) *{\cdots};
( 10, 13) *{a_k};
(-25,-2); (-20,0) **@{-};
(-25, 0); (-20,0) **@{-};
(-25, 2); (-20,0) **@{-};
(25,-2); (20,0) **@{-};
(25, 0); (20,0) **@{-};
(25, 2); (20,0) **@{-};
\end{xy}$
&\qquad&
$\begin{xy}
(-20,0) *{\bullet}; (20,0) *{\bullet} **@{-};
(-20,  3) *{b};
( 20,  3) *{c};
(-12,0)*{\bullet};
(-8,0)*{\bullet};
(-4,0)*{\bullet};
( 0,0)*{\bullet};
( 4,0)*{\bullet};
( 8,0)*{\bullet};
(12,0)*{\bullet};
(-11.5,-3) *{x_1};
(-7.5,-3) *{x_2};
(1.5,-3) *{\cdots};
(10,-3) *!L{x_{2k+1}};
(-25,-2); (-20,0) **@{-};
(-25, 0); (-20,0) **@{-};
(-25, 2); (-20,0) **@{-};
(25,-2); (20,0) **@{-};
(25, 0); (20,0) **@{-};
(25, 2); (20,0) **@{-};
\end{xy}$\\[3mm]
(a) $\Gamma_1$ &&
(b) $\Gamma_2$\\[5mm]
$\begin{xy}
(0,10) *{\bullet}; (-20,0) *{\bullet}**@{-};
(0,10); (20,0) *{\bullet} **@{-};
(-20,0); (20,0) *{\bullet} **@{-};
(-20, 3) *{b};
( 20, 3) *{c};
(3,12) *!L{x_1^{x_2x_3\cdots x_{2k+1}}};
(0,10); (-12,0)*{\bullet}**@{-};
(0,10); (-8,0)*{\bullet}**@{-};
(0,10); (-4,0)*{\bullet}**@{-};
(0,10); ( 0,0)*{\bullet}**@{-};
(0,10); ( 4,0)*{\bullet}**@{-};
(0,10); ( 8,0)*{\bullet}**@{-};
(0,10); (12,0)*{\bullet}**@{-};
(-11.5,-3) *{x_1};
(-7.5,-3) *{x_2};
(1.5,-3) *{\cdots};
(10,-3) *!L{x_{2k+1}};
(-25,-2); (-20,0) **@{-};
(-25, 0); (-20,0) **@{-};
(-25, 2); (-20,0) **@{-};
(25,-2); (20,0) **@{-};
(25, 0); (20,0) **@{-};
(25, 2); (20,0) **@{-};
\end{xy}$
&&
$\begin{xy}
(0,10) *{\bullet}; (-20,0) *{\bullet}**@{-};
(0,10); (20,0) *{\bullet} **@{-};
(-20, 3) *{b};
( 20, 3) *{c};
(3,12) *!L{x_1^{x_2x_3\cdots x_{2k+1}}};
(0,10); (-10,0)*{\bullet}**@{-};
(0,10); (-2,0)*{\bullet}**@{-};
(0,10); (10,0)*{\bullet}**@{-};
(-10,-3) *{x_2};
(-2,-3) *{x_4};
(4,-3) *{\cdots};
(11,-3) *{x_{2k}};
(-25,-2); (-20,0) **@{-};
(-25, 0); (-20,0) **@{-};
(-25, 2); (-20,0) **@{-};
(25,-2); (20,0) **@{-};
(25, 0); (20,0) **@{-};
(25, 2); (20,0) **@{-};
\end{xy}$\\[3mm]
(c) Subgraph of $\Gamma_2^E$ induced by $S_1$&&
(d) Subgraph of $\Gamma_2^E$ induced by $S_2$&&
\end{tabular}
\caption{The graph $\Gamma_1$ embeds into $\Gamma_2^E$ as an induced subgraph.}
\label{fig:deg1k}
\end{figure}

\begin{proof}
By Lemmas~\ref{lem:comm1} and~\ref{lem:comm2},
the element $x_1^{x_2x_3\cdots x_{2k+1}}$ in $G(\Gamma_2)$
commutes with $v\in V(\Gamma_2)$
if and only if $v\not\in\{b, c, x_1,x_2,\ldots,x_{2k+1}\}$.
Hence the subgraph of $\Gamma_2^E$ induced by
$S_1=V(\Gamma_2)\cup\{x_1^{x_2x_3\cdots x_{2k+1}}\}$
is as in Figure~\ref{fig:deg1k}(c).
The subgraph of $\Gamma_2^E$ induced by
$S_2=S_1\setminus\{x_1,x_3,\ldots,x_{2k+1}\}$
is as in Figure~\ref{fig:deg1k}(d), which is isomorphic to $\Gamma_1$.
Therefore $\Gamma_1\le\Gamma_2^E$ and hence $G(\Gamma_1)\le G(\Gamma_2)$.
\end{proof}

\begin{proposition}\label{prop:deg3}
Let $\Gamma_1$ be a finite   graph containing a degree 3 vertex $x$
with $\Lk(x)=\{a,b,c\}$ as in Figure~\ref{fig:deg3}(a).
Let $\Gamma_2$ be the graph obtained from $\Gamma_1$
by replacing the tripod centered at $x$ with a 6-cycle as in Figure~\ref{fig:deg3}(b),
where each vertex $v\in V(\Gamma_1)\setminus \{x\}$ is renamed as $v_1\in V(\Gamma_2)$.
Then $G(\Gamma_1)\le G(\Gamma_2)$.
\end{proposition}

\begin{figure}
\begin{tabular}{*5c}
$\begin{xy}/r1.4mm/:
(0,0) *{\bullet}; (10,0) *{\bullet} **@{-};
(0, -3) *{a};
(10, -3) *{x};
(17, 10) *{b};
(17,-10) *{c};
(10,0); (17,7) *{\bullet} **@{-};
(10,0); (17,-7) *{\bullet} **@{-};
(-5,-2); (0,0) **@{-};
(-5, 0); (0,0) **@{-};
(-5, 2); (0,0) **@{-};
(22,9); (17,7) **@{-};
(22,7); (17,7) **@{-};
(22,5); (17,7) **@{-};
(22,-9); (17,-7) **@{-};
(22,-7); (17,-7) **@{-};
(22,-5); (17,-7) **@{-};
\end{xy}$
&\qquad\qquad\qquad&
$\begin{xy}/r1.4mm/:
(0,0);
(9,2) *{\bullet} **@{-};
(17, 7) *{\bullet} **@{-};
(12,0) *{\bullet} **@{-};
(17,-7) *{\bullet} **@{-};
(9,-2) *{\bullet} **@{-};
(0,0) *{\bullet} **@{-};
(0, -3) *{a_1};
(17, 10) *{b_1};
(17,-10) *{c_1};
(14,  0) *!L{x_1};
(9, -4) *!U{x_2};
(9,  4) *!D{x_3};
(-5,-2); (0,0) **@{-};
(-5, 0); (0,0) **@{-};
(-5, 2); (0,0) **@{-};
(22,9); (17,7) **@{-};
(22,7); (17,7) **@{-};
(22,5); (17,7) **@{-};
(22,-9); (17,-7) **@{-};
(22,-7); (17,-7) **@{-};
(22,-5); (17,-7) **@{-};
\end{xy}$\\[3mm]
(a) $\Gamma_1$ &&
(b) $\Gamma_2$
\end{tabular}
\caption{$G(\Gamma_1)$ embeds into $G(\Gamma_2)$.}
\label{fig:deg3}
\end{figure}
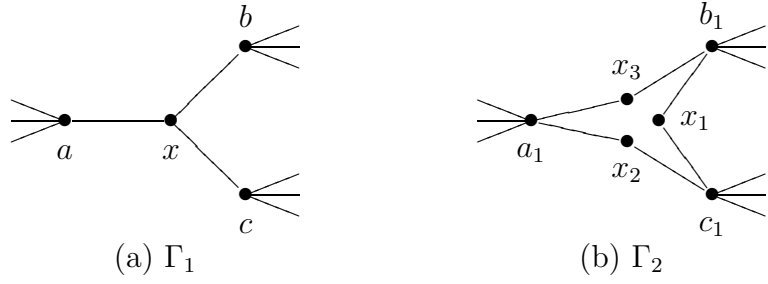

\begin{proof}
Let $\phi:\Gamma_2\to\Gamma_1$ be the graph homomorphism
defined by $\phi(v_1)=v$ for $v_1\not\in\{x_1,x_2,x_3\}$ and $\phi(x_i)=x$ for $i=1,2,3$.
Then $\phi^*:G(\Gamma_1)\to G(\Gamma_2)$ is the group homomorphism
such that $\phi^*(x)=x_1x_2x_3$ and $\phi^*(v)=v_1$ for $v\ne x$.
We will show that $\phi^*$ is injective.

\medskip
Let $\Gamma_1'=\Gamma_1\setminus a$, $\Gamma_2'=\Gamma_2\setminus a_1$
and $\phi_1=\phi|_{\Gamma_2'}:\Gamma_2'\to\Gamma_1'$.
See Figure~\ref{fig:deg3-C1}.

\medskip
\noindent{\textbf Claim 1.}
$\phi_1$ is $v_1$-surviving for all $v_1\in V(\Gamma_2')\setminus\{x_2,x_3\}$.
In particular, $\phi_1^*$ is injective.

\begin{proof}[Proof of Claim 1]
Let $\Gamma_2''=\Gamma_2'\setminus\{x_2,x_3\}$.
Let $\iota: \Gamma_2''\to\Gamma_2'$ be the inclusion.
Then $\iota^*:G(\Gamma_2')\to G(\Gamma_2'')$ is an epimorphism
such that $\iota^*(v_1)=v_1$ for all $v_1\ne x_2,x_3$ and $\iota^*(x_2)=\iota^*(x_3)=1$.

On the other hand, $\phi_1\circ\iota:\Gamma_2''\to \Gamma_1'$ is a graph isomorphism
sending $v_1$ to $v$ for each $v_1\in V(\Gamma_2'')$,
hence $\iota^*\circ\phi_1^*: G(\Gamma_1')\to G(\Gamma_2'')$ is
an isomorphism sending $w(x,b,c,\ldots)$ to $w(x_1,b_1,c_1,\ldots)$.
In particular, if $w(x,b,c,\ldots)$ is a reduced word in $G(\Gamma_1')$,
then $w(x_1,b_1,c_1,\ldots)$ is also a reduced word in $G(\Gamma_2'')$.

Assume that $\phi_1$ is not $v_1$-surviving for some $v_1\in V(\Gamma_2')\setminus\{x_2,x_3\}$.
Then there exists a nontrivial reduced word $w=w(x,b,c,\ldots)$ in $G(\Gamma'_1)$
such that the word
$$\phi_1^*(w)=w(x_1x_2x_3,b_1,c_1,\ldots)$$
has a cancellation of $v_1$.
This implies that
$$\iota^*(\phi_1^*(w))=w(x_1,b_1,c_1,\ldots)$$
also has a cancellation of $v_1$.
This is a contradiction because $w(x_1,b_1,c_1,\ldots)$ is a reduced word
in $G(\Gamma_2'')$.
Therefore
$\phi_1$ is $v_1$-surviving for all $v_1\in V(\Gamma_2')\setminus\{x_2,x_3\}$.

If $w\in G(\Gamma_1')$ is a nontrivial element, then
$v\in\supp(w)$ for some $v\in V(\Gamma_1')$,
hence $v_1\in\supp(\phi_1^*(w))$ because $\phi_1$ is $v_1$-surviving.
This implies that $\phi_1^*$ is injective.
\end{proof}

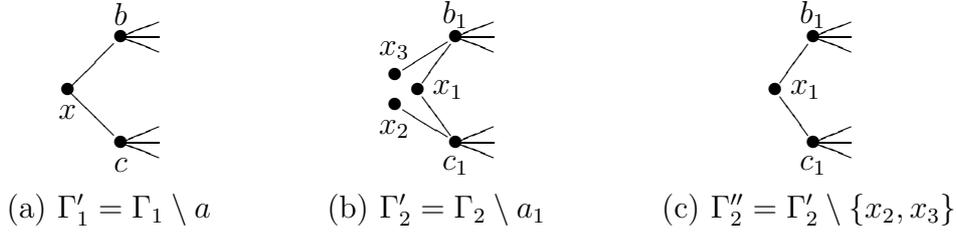
\begin{figure}
\begin{tabular}{*8c}
$\begin{xy}
(10,0) *{\bullet};
(10, -3) *{x};
(17, 10) *{b};
(17,-10) *{c};
(10,0); (17,7) *{\bullet} **@{-};
(10,0); (17,-7) *{\bullet} **@{-};
(22,9); (17,7) **@{-};
(22,7); (17,7) **@{-};
(22,5); (17,7) **@{-};
(22,-9); (17,-7) **@{-};
(22,-7); (17,-7) **@{-};
(22,-5); (17,-7) **@{-};
\end{xy}$
&\qquad\qquad&
$\begin{xy}
(9,2) *{\bullet}; (17, 7) *{\bullet} **@{-};
(12,0) *{\bullet} **@{-};
(17,-7) *{\bullet} **@{-};
(9,-2) *{\bullet} **@{-};
(17, 10) *{b_1};
(17,-10) *{c_1};
(14,  0) *!L{x_1};
(9, -4) *!U{x_2};
(9,  4) *!D{x_3};
(22,9); (17,7) **@{-};
(22,7); (17,7) **@{-};
(22,5); (17,7) **@{-};
(22,-9); (17,-7) **@{-};
(22,-7); (17,-7) **@{-};
(22,-5); (17,-7) **@{-};
\end{xy}$
&\qquad\qquad&
$\begin{xy}
(17, 7) *{\bullet};
(12,0) *{\bullet} **@{-};
(17,-7) *{\bullet} **@{-};
(17, 10) *{b_1};
(17,-10) *{c_1};
(14,  0) *!L{x_1};
(22,9); (17,7) **@{-};
(22,7); (17,7) **@{-};
(22,5); (17,7) **@{-};
(22,-9); (17,-7) **@{-};
(22,-7); (17,-7) **@{-};
(22,-5); (17,-7) **@{-};
\end{xy}$\\[10mm]
(a) $\Gamma_1'=\Gamma_1\setminus a$ &&
(b) $\Gamma_2'=\Gamma_2\setminus a_1$ &&
(c) $\Gamma_2''=\Gamma_2'\setminus \{x_2,x_3\}$
\end{tabular}
\caption{Graphs $\Gamma_1'$, $\Gamma_2'$ and $\Gamma_2''$.}
\label{fig:deg3-C1}
\end{figure}

\medskip
\noindent{\textbf Claim 2.}
For $w\in G(\Gamma_1')$,
if $x\in\supp(w)$ then
either $x_2$ or $x_3$ belongs to $\supp(\phi_1^*(w))$.

\begin{proof}[Proof of Claim 2]
We may assume that $w=w(x,b,c,\ldots)$ is expressed as a reduced word
as follows:
\begin{align}\label{eq:dec}
w
&=w_1 \cdot x^k \cdot t \cdot w_2
=w_1(b,c,\ldots)\cdot x^k \cdot t \cdot w_2(x,b,c,\ldots), 
\end{align}
where $x\not\in\supp(w_1)$, $k\ne 0$ and $t\in\{b,c\}$.
(The word $w_1$ is possibly empty.)
This decomposition is obtained as follows.

Decompose the element $w$ as $w=w_1u_1$ such that the word $w_1u_1$
is reduced and $x\not\in\supp(w_1)$.
We may assume that $w_1$ has the largest word length
among all such decompositions.
Then $u_1$ must start with $x$ or $x^{-1}$.
Let $u_1=x^k u_2$ for $k\ne 0$.
Take $|k|$ as large as possible. Then $u_2$ cannot start with $x$ or $x^{-1}$.
Moreover, $u_2$ cannot start with a letter $y$ such that $[y,x]=1$,
for otherwise we can make $w_1$ longer.
Therefore either $u_2=1$ or $u_2$ starts with $b$ or $c$.
If $u_2=1$, then $w_1(b_1,c_1,\ldots)x_1^kx_2^kx_3^k$
is a reduced word representing $\phi^*(w)=\phi^*(w_1(b,c,\ldots)x^k)$,
hence $x_2,x_3\in\supp(\phi^*(w))$.
If $u_2$ starts with $b$ or $c$, then we have
the desired decomposition of $w$ as Eq.\ (\ref{eq:dec}).

Without loss of generality, we may assume $t=c$.
Then
\begin{align*}
w &= w_1 \cdot x^k \cdot c\cdot w_2
  = w_1(b,c,\ldots) \cdot x^k \cdot c\cdot w_2(x,b,c,\ldots), \\
\phi_1^*(w)&=\phi_1^*(w_1) \cdot (x_1x_2x_3)^k \cdot c_1 \cdot \phi_1^*(w_2)\\
 & = w_1(b_1,c_1,\ldots) \cdot x_1^kx_2^k  \cdot c_1 \cdot x_3^k \cdot \phi_1^*(w_2).
\end{align*}

Let $w_3$ be a reduced word in $G(\Gamma_2')$ representing $x_3^k \cdot \phi_1^*(w_2)$.
Then
$$
\phi_1^*(w)=w_1(b_1,c_1,\ldots)x_1^kx_2^k c_1 w_3.
$$

Let $w'$ be the word $w_1(b_1,c_1,\ldots)x_1^kx_2^k c_1 w_3$ in the above.
Assume that $w'$ has a cancellation of $x_2$.
Since $w_3$ is reduced, the cancellation must occur between
$x_2^{\pm 1}$ in $x_2^k$ and $x_2^{\mp 1}$ in $w_3$, hence
$w'$ has a subword
$$
x_2^{\pm 1}c_1w_4x_2^{\mp1},
$$
where $w_4x_2^{\mp1}$ is an initial subword of $w_3$
and $\supp(c_1w_4)\cap \Lk(x_2)=\emptyset$.
Since $\phi_1$ is $c_1$-surviving by Claim 1, we have $c_1\in\supp(c_1w_4)$.
Since $c_1\in\Lk(x_2)$, this contradicts $\supp(c_1w_4)\cap \Lk(x_2)=\emptyset$.
Therefore $w'$ has no cancellation of $x_2$,
hence $x_2\in \supp(w')=\supp(\phi_1^*(w))$.
\end{proof}

\medskip
\noindent{\textbf Claim 3.}
$\phi$ is $a_1$-surviving.
\begin{proof}[Proof of Claim 3]
Assume that $\phi$ is not $a_1$-surviving.
Then there exists a reduced word $w$ in $G(\Gamma_1)$ such that
$\phi^*(w)$ has an innermost cancellation of $a_1$.
Hence the word $w$ has a subword
$$a^{\pm 1}w_1a^{\mp 1}$$
such that $w_1$ is a nontrivial reduced word in $G(\Gamma_1\setminus a)=G(\Gamma_1')$
with $\supp(\phi^*(w_1))\cap \Lk(a_1)=\emptyset$.
Notice that $\phi^*(w_1)=\phi_1^*(w_1)$ and
that $\supp(w_1)\cap \Lk(a)\ne\emptyset$ because $w$ is a reduced word.

\medskip
Assume that there exists $v\ne x$ in $\supp(w_1)\cap \Lk(a)$.
Then $v_1\in\supp(\phi^*(w_1))=\supp(\phi_1^*(w_1))$ because
$\phi_1$ is $v_1$-surviving by Claim 1
and $w_1$ is a word in $G(\Gamma_1')$ with $v\in\supp(w_1)$.
Since $v_1\in\supp(\phi^*(w_1))\cap \Lk(a_1)$, this contradicts
$\supp(\phi^*(w_1))\cap \Lk(a_1)=\emptyset$.

Therefore $\supp(w_1)\cap \Lk(a)=\{x\}$.
Since $\supp(\phi^*(w_1))\cap \Lk(a_1)=\emptyset$,
the set $\supp(\phi^*(w_1))$ contains neither $x_2$ nor $x_3$.
This is impossible by Claim 2. Therefore $\phi$ is $a_1$-surviving.
\end{proof}

Let $w$ be a nontrivial reduced word in $G(\Gamma_1)$.
If $a\in \supp(w)$, then $\phi^*(w)\ne 1$
because $\phi$ is $a_1$-surviving by Claim 3.
If $a\not\in\supp(w)$, then $w$ is a nontrivial reduced word in $G(\Gamma_1')$, hence
$\phi^*(w)=\phi_1^*(w)$ is nontrivial by Claim 1.
Therefore $\phi^*$ is injective.
\end{proof}

\begin{theorem}\label{thm:A}
$G(T_2)$ embeds into $G(P_{22})$.
\end{theorem}

\begin{proof}
Let $\Gamma_1$ and $\Gamma_2$ be the graphs in Figure~\ref{fig:T2}(b,c).
Then $G(T_2)$ embeds into $G(\Gamma_1)$ by Proposition~\ref{prop:deg3}
and $G(\Gamma_1)$ embeds into $G(\Gamma_2)$ by Proposition~\ref{prop:deg1k}.
Notice that $\Gamma_2$ is the cycle $C_{12}$.
By Theorem 3.16 in~\cite{LL16}, $G(C_m)$ embeds into $G(P_{2m-2})$ for all $m\ge 3$.
In particular, $G(C_{12})$ embeds into $G(P_{22})$.
Consequently, $G(T_2)$ embeds into $G(P_{22})$.
\end{proof}

\begin{figure}
\begin{tabular}{*8c}
$\begin{xy}/r.9mm/:
(-10,0) *{\bullet}; (0,0) *{\bullet} **@{-}; (10,0) *{\bullet} **@{-};
(17,7) *{\bullet} **@{-}; (24,14) *{\bullet} **@{-};
(10,0); (17,-7) *{\bullet} **@{-}; (24,-14) *{\bullet} **@{-};
(0, -3) *{a};
(10, -3) *{x};
(-10, -3) *{p};
(19,-5) *{c};
(19, 5) *{b};
(26,12) *{q};
(26,-12) *{r};
\end{xy}$
&\quad&
$\begin{xy}/r.9mm/:
(-10,0) *{\bullet}; (0,0) *{\bullet} **@{-};
(9,2) *{\bullet} **@{-}; (17, 7) *{\bullet} **@{-}; (24,14) *{\bullet} **@{-};
(17,7); (12,0) *{\bullet} **@{-}; (17,-7) *{\bullet} **@{-}; (24,-14) *{\bullet} **@{-};
(17,-7); (9,-2) *{\bullet} **@{-};
(0,0) *{\bullet} **@{-};
(0, -3) *{a};
(19,-5) *{c};
(19.5, 6) *{b};
(14,  0) *!L{x_1};
(9, -4) *!U{x_2};
(9,  4) *!D{x_3};
(-10, -3) *{p};
(26,12) *{q};
(26,-12) *{r};
\end{xy}$
&\quad&
$\begin{xy}/r.9mm/:
(-10,0) *{\bullet}; (0,1.5) *{\bullet} **@{-};
(9,4) *{\bullet} **@{-}; (17, 9) *{\bullet} **@{-}; (25,14) *{\bullet} **@{-};
(19,7) *{\bullet} **@{-}; (13,0) *{\bullet} **@{-};
(19,-7) *{\bullet} **@{-}; (25,-14) *{\bullet} **@{-};
(17,-9) *{\bullet} **@{-}; (9,-4) *{\bullet} **@{-};
(0,-1.5) *{\bullet} **@{-}; (-10,0) *{\bullet} **@{-};
(0,  4.5) *{a_3};
(-10, -3) *{a_2};
(0, -4.5) *{a_1};
(15,12) *{b_1};
(28,12) *{b_2};
(22, 5) *{b_3};
(22,-5) *{c_1};
(28,-12) *{c_2};
(15,-11) *{c_3};
(17, 0) *{x_1};
(8.5,-7) *{x_2};
(8.5, 7.5) *{x_3};
\end{xy}$\\
(a) $T_2$ &&
(b) $\Gamma_1$ &&
(c) $\Gamma_2=C_{12}$
\end{tabular}
\caption{Graphs $T_2$, $\Gamma_1$ and $\Gamma_2$: $G(T_2)\le G(\Gamma_1)\le G(\Gamma_2)$.}
\label{fig:T2}
\end{figure}
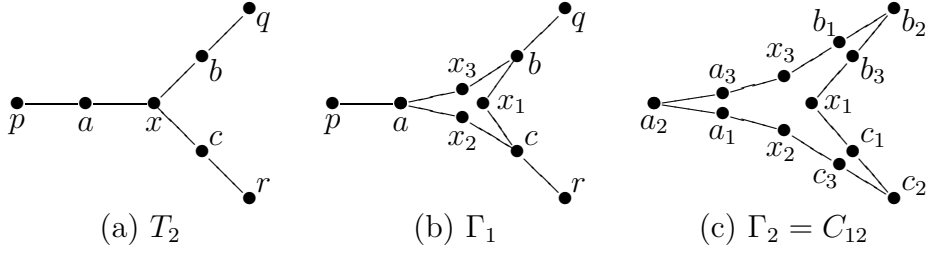

\section{Embeddability into extension graphs}

In this section, we show that $T_2\not\le P_n^E$ for any $n$,
and then give a characterization of trees that embeds into $P_n^E$ for some $n$
as an induced subgraph.
Let $\Gamma$ be a finite   graph.

\begin{lemma}\label{lem:ind}
Let $A\subset V(\Gamma)$ and $b^w\in V(\Gamma^E)$,
where $b\in V(\Gamma)$ and $w\in G(\Gamma)$.
Suppose that there is no edge between $b^w$ and the vertices in $A$.
Then there is an inner automorphism $\psi$ of $G(\Gamma)$ such that
$\psi(a)=a$ for all $a\in A$ and $\psi(b^w)=b$.
(See Figure~\ref{fig:push}.)
\end{lemma}

\begin{proof}
We may assume that $w^{-1}bw$ is reduced.
Let $a\in A$.
As $[b^w,a]=1$, $a$ commutes with each element of $\supp(b^w)=\{b\}\cup\supp(w)$,
hence $a$ commutes with $w$.
Let $\psi$ be the inner automorphism of $G(\Gamma)$
sending $g\in G(\Gamma)$ to $wgw^{-1}$, then
$\psi(a)=a$ for all $a\in A$ and $\psi(b^w)=b$.
\end{proof}

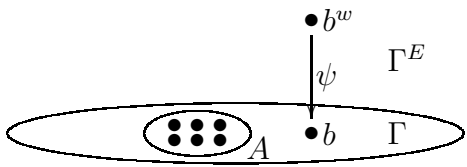
\begin{figure}
$$\begin{xy}
(20,15) *{\bullet} *+!L{b^w};
(20, 0) *{\bullet} *+!L{b};
(20,13); (20,2) **@{-}    ?>*@{>}   ?(0.5) *!/_2mm/{\psi};
(5, 0) *\xycircle(7,3){}; (13,-2) *{A};
(5,1) *{\bullet};  (5,-1) *{\bullet};
(2,1) *{\bullet};  (2,-1) *{\bullet};
(8,1) *{\bullet};  (8,-1) *{\bullet};
(10, 0) *\xycircle(30,4){};
( 30,0) *!L{\Gamma};
( 30,10) *!L{\Gamma^E};
\end{xy}$$
\caption{An inner automorphism $\psi$ sends $b^w$ into $\Gamma$ fixing $A$ pointwise.}
\label{fig:push}
\end{figure}

For a graph $\Lambda$, a subset $A$ of $V(\Lambda)$ is called
an \emph{independent set} if there is no edge between any two vertices in $A$.
Any independent subset $A$ of $V(\Gamma^E)$ is a finite set
because $G(\Gamma)$ contains a free abelian subgroup of rank $|A|$
by the Extension Graph Theorem
and because the maximum rank of a free abelian subgroup of $G(\Gamma)$
is the size of the largest independent subset of $V(\Gamma)$.

The following corollary shows that an independent subset of $V(\Gamma^E)$
is a conjugate of an independent subset of $V(\Gamma)$ by an element of $G(\Gamma)$.
Similar arguments were used in~\cite{KK13} and~\cite{CDK13}.

\begin{corollary}\label{lem:indep}
Let $A\subset V(\Gamma^E)$ be an independent set.
Then there is an inner automorphism $\psi$ of $G(\Gamma)$ such that
$\psi(A)\subset V(\Gamma)$.
\end{corollary}

\begin{proof}
Let $A=\{v_1,\ldots,v_m\}\subset V(\Gamma^E)$.
By Lemma~\ref{lem:ind},
if $v_1,\ldots,v_{k-1}\in V(\Gamma)$ for $1\le k\le m$,
then there exists an inner automorhpism $\psi$ of $G(\Gamma)$ such that
$\psi(v_j)=v_j$ for $1\le j\le k-1$ and $\psi(v_k)\in V(\Gamma)$.
Since the composition of inner automorphisms is also an inner automorphism,
we are done by using induction on $|A|$.
\end{proof}

\begin{lemma}\label{lem:path}
Let $\{x,p,q\}$ be an independent subset of\/ $V(P_n)$ for some $n\ge 5$
such that $p$ lies between $x$ and $q$.
Let $b^w\in V(P_n^E)$ with $[b^w,p]=1$ in $G(P_n)$,
where $b\in V(P_n)$ and $w\in G(P_n)$.
Then either $[b^w,x]=1$ or $[b^w,q]=1$.
\end{lemma}

\begin{proof}
The graph $P_n\setminus \Lk(p)$ has three path components,
say $P_n\setminus \Lk(p)=\Gamma_1\cup\{p\}\cup\Gamma_2$,
where $\Gamma_1$ (resp.\ $\Gamma_2$) is the path component
containing $x$ (resp.\ $q$).
Note that two vertices from distinct components commute with each other
in $G(P_n)$.

As $[b^w,p]=1$, one has $\supp(b^w)\cap \Lk(p)=\emptyset$ by Lemma~\ref{lem:comm1}.
Since $\supp(b^w)$ spans a connected subgraph of $P_n$ by Lemma~\ref{lem:comm2},
$\supp(b^w)$ is contained in one of $\Gamma_1$, $\Gamma_2$ and $\{p\}$.
Therefore $b^w$ commutes with either $x$ or $q$.
\end{proof}

\begin{theorem}\label{thm:B}
The tripod $T_2$ does not embed into $P_n^E$ as an induced subgraph for any $n$.
\end{theorem}

\begin{proof}
Assume $T_2\le P_n^E$ for some $n$.
Let $\phi:T_2\to P_n^E$ be the embedding.
Label the vertices of $T_2$ as in Figure~\ref{fig:T2}(a).
Let $v_x, v_p, v_q, v_r$ denote the images of $x,p,q,r$ under $\phi$, respectively.
Since $\{x,p,q,r\}$ is an independent subset of $V(T_2)$,
$\{v_x,v_p,v_q,v_r\}$ is also an independent subset of $V(P_n^E)$.
By Corollary~\ref{lem:indep},
we may assume that $\{v_x,v_p,v_q,v_r\}\subset V(P_n)$.

Since $P_n$ is a path graph, at least two of $v_p, v_q, v_r$ are in the same
component of $P_n\setminus v_x$.
Without loss of generality, we may assume that $v_p$ and $v_q$ are in the same component.
Moreover, we may assume that $v_x, v_p, v_q$ lie in $P_n$ in this order.

For the vertex $b$ of $T_2$ in Figure~\ref{fig:T2}(a),
let $\phi(b)=v_b^{w_b}$, where $v_b\in V(P_n)$ and $w_b\in G(P_n)$.
Since $[b,p]=1$, we have $[v_b^{w_b},v_p]=1$.
By Lemma~\ref{lem:path}, either $[v_b^{w_b},v_x]=1$ or $[v_b^{w_b},v_q]=1$.
This contradicts that $[b,x]\ne 1$ and $[b,q]\ne 1$.
Therefore $T_2\not\le P_n^E$.
\end{proof}

The proof of Theorem~\ref{thm:B} uses only the following properties
of the tripod $T_2$: in Figure~\ref{fig:T2}(a),
(i) $\{x,p,q,r\}$ is an independent subset of $V(T_2)$;
(ii) the vertex $a$ (resp.\ $b$, $c$) is adjacent to neither $q$ nor $r$
(resp.\ neither $p$ nor $r$, neither $p$ nor $q$).
Therefore, by the same proof of Theorem~\ref{thm:B},
none of the following graphs
embeds into $P_n^E$ as an induced subgraph for any $n$.
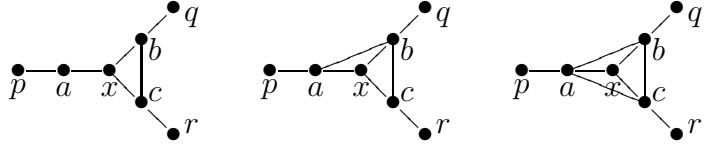
$$
\begin{xy}/r.6mm/:
(-10,0) *{\bullet}; (0,0) *{\bullet} **@{-}; (10,0) *{\bullet} **@{-};
(17,7) *{\bullet} **@{-}; (24,14) *{\bullet} **@{-};
(10,0); (17,-7) *{\bullet} **@{-}; (24,-14) *{\bullet} **@{-};
(0,  -4) *{a};
(10, -4) *{x};
(-10,-4) *{p};
(20, -5) *{c};
(20,  5) *{b};
(28, 12) *{q};
(28,-12) *{r};
(17,7); (17,-7) **@{-};
\end{xy}
\qquad
\begin{xy}/r.6mm/:
(-10,0) *{\bullet}; (0,0) *{\bullet} **@{-}; (10,0) *{\bullet} **@{-};
(17,7) *{\bullet} **@{-}; (24,14) *{\bullet} **@{-};
(10,0); (17,-7) *{\bullet} **@{-}; (24,-14) *{\bullet} **@{-};
(0,  -4) *{a};
(10, -4) *{x};
(-10,-4) *{p};
(20, -5) *{c};
(20,  5) *{b};
(28, 12) *{q};
(28,-12) *{r};
(0,0); (17,7) **@{-}; (17,-7) **@{-};
\end{xy}
\qquad
\begin{xy}/r.6mm/:
(-10,0) *{\bullet}; (0,0) *{\bullet} **@{-}; (10,0) *{\bullet} **@{-};
(17,7) *{\bullet} **@{-}; (24,14) *{\bullet} **@{-};
(10,0); (17,-7) *{\bullet} **@{-}; (24,-14) *{\bullet} **@{-};
(0,  -4) *{a};
(10, -4) *{x};
(-10,-4) *{p};
(20, -5) *{c};
(20,  5) *{b};
(28, 12) *{q};
(28,-12) *{r};
(0,0); (17,7) **@{-}; (17,-7) **@{-}; (0,0) **@{-};
\end{xy}
$$

By the above theorem together with Theorem~\ref{thm:A},
we obtain the following.

\begin{corollary}\label{colo:main}
There exist a finite tree $T$ and a finite path graph $P$ such that
$G(T)$ embeds into $G(P)$ but $T$ does not embed into
$P^E$ as an induced subgraph.
\end{corollary}

\begin{definition}
A finite tree $T$ is called a \emph{hairy path graph}
if $T$ contains a path graph $P_m$ as an induced subgraph
such that each vertex of $V(T)\setminus V(P_m)$ is adjacent to a vertex of $P_m$
as in Figure~\ref{fig:hairy}.
\end{definition}

\begin{figure}
$$\begin{xy}/r1.2mm/:
(-30,0) *{\bullet};  (30,0) *{\bullet} **@{-};
(-20,0) *{\bullet};
(-10,0) *{\bullet};
(  0,0) *{\bullet};
( 10,0) *{\bullet};
( 20,0) *{\bullet};
(-20,0); (-20,10) *{\bullet} **@{-};
(-20,0); (-22,10) *{\bullet} **@{-};
(-20,0); (-18,10) *{\bullet} **@{-};
(-10,0); (-13,10) *{\bullet} **@{-};
(-10,0); (-11,10) *{\bullet} **@{-};
(-10,0); ( -9,10) *{\bullet} **@{-};
(-10,0); ( -7,10) *{\bullet} **@{-};
( 0,0); (0,10) *{\bullet} **@{-};
( 0,0); (-2,10) *{\bullet} **@{-};
( 0,0); (2,10) *{\bullet} **@{-};
( 10,0); (  9,10) *{\bullet} **@{-};
( 10,0); ( 11,10) *{\bullet} **@{-};
( 20,0); ( 17,10) *{\bullet} **@{-};
( 20,0); ( 19,10) *{\bullet} **@{-};
( 20,0); ( 21,10) *{\bullet} **@{-};
( 20,0); ( 23,10) *{\bullet} **@{-};
(-30,-3) *{v_1};
(-20,-3) *{v_2};
(-10,-3) *{v_3};
(  5,-3) *{\cdots};
( 20,-3) *{v_{m-1}};
( 30,-3) *{v_m};
\end{xy}$$
\caption{Hairy path graph}
\label{fig:hairy}
\end{figure}
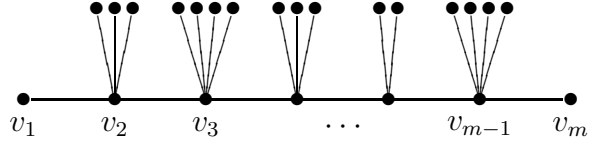

Let $T$ be the hairy path graph in Figure~\ref{fig:ex1}(a).
Label the vertices of the path graph $P_{10}$ as in Figure~\ref{fig:ex1}(b).
By the same argument in the proof of Proposition~\ref{prop:deg1k},
the subgraph of $P_{10}^E$ induced by
$S=\{x_1^{x_2x_3}, y_1^{y_2\cdots y_5}, b, c, x_2, y_2, y_4\}$
is isomorphic to $T$ as in Figure~\ref{fig:ex1}(c).
Therefore $T$ embeds into $P_{10}^E$ as an induced subgraph.

Using Theorem~\ref{thm:B},
we obtain a characterization of trees
that embeds into $P_n^E$  as an induced subgraph.

\begin{figure}
\begin{tabular}{*9c}
$\begin{xy} /r.8mm/:
(0,0) *{\bullet};
(10,0) *{\bullet} **@{-};
(20,0) *{\bullet} **@{-};
(30,0) *{\bullet} **@{-};
(10, 0); (10,7) *{\bullet} **@{-};
(20, 0); (18,7) *{\bullet} **@{-};
(20, 0); (22,7) *{\bullet} **@{-};
(0, -3) *{b};
(10, -3) *{x};
(20, -3) *{y};
(30, -3) *{c};
\end{xy}$
&\qquad&
$\begin{xy}/r.9mm/:
(0,0) *{\bullet};
(7,0) *{\bullet} **@{-};
(11,0) *{\bullet} **@{-};
(15,0) *{\bullet} **@{-};
(22,0) *{\bullet} **@{-};
(25,0) *{\bullet} **@{-};
(28,0) *{\bullet} **@{-};
(31,0) *{\bullet} **@{-};
(34,0) *{\bullet} **@{-};
(41,0) *{\bullet} **@{-};
(0, -3) *{b};
(7, -3) *{x_1};
(11, -3) *{x_2};
(15, -3) *{x_3};
(22, -3) *{y_1};
(28, -3) *{\cdots};
(34, -3) *{y_5};
(41, -3) *{c};
\end{xy}$
&\qquad&
$\begin{xy}/r.9mm/:
(0,0) *{\bullet};
(11,7) *{\bullet}**@{-};
(11,0) *{\bullet}**@{-};
(28,7)*{\bullet}; (11,7) *{\bullet}**@{-};
(28,7); (25,0) *{\bullet}**@{-};
(28,7); (31,0) *{\bullet}**@{-};
(28,7); (41,0) *{\bullet}**@{-};
(0, -3) *{b};
(11, -3) *{x_2};
(25, -3) *{y_2};
(31, -3) *{y_4};
(41, -3) *{c};
(11, 10) *!L{x_1^{x_2x_3}};
(28, 10) *!L{y_1^{y_2 y_3 y_4 y_5}};
\end{xy}$\\[3mm]
(a) $T$ &&
(b) $P_{10}$ &&
(c) An induced subgraph of $P_{10}^E$
\end{tabular}
\caption{$T$ embeds into $P_{10}^E$ as an induced subgraph.}\label{fig:ex1}
\end{figure}
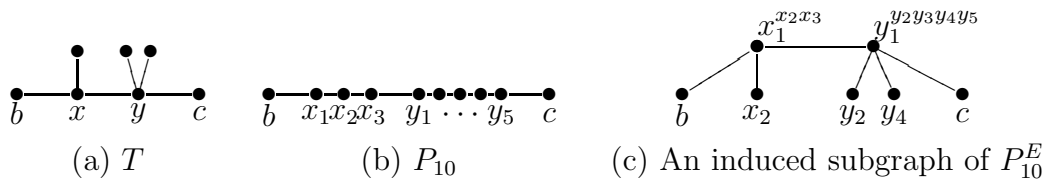

\begin{theorem}\label{thm:hairy}
For a finite tree $T$, the following are equivalent.
\begin{enumerate}
\item[(i)] $T\le P_n^E$ for some $n$.
\item[(ii)] $T_2\not\le T$.
\item[(iii)] $T$ is a hairy path graph.
\end{enumerate}
\end{theorem}

\begin{proof}
(i) $\Rightarrow$ (ii)\ \
It follows from Theorem~\ref{thm:B}.

\smallskip
(ii) $\Rightarrow$ (iii)\ \
Let $P_m$ be a longest path graph among induced subgraphs of $T$.
Let $V(P_m)=\{v_1,\ldots,v_m\}$ such that $v_i$ and $v_{i+1}$ are adjacent for $i=1,\ldots,m-1$.
Let $v\in V(T)\setminus V(P_m)$.
We will show that $v$ is adjacent to $v_i$ for some $2\le i\le m-1$,
hence $T$ is a hairy path graph.

Since $T$ is a tree, there exists a unique $v_i\in V(P_m)$ that is nearest to $v$.
Since $P_m$ is longest, $i\not\in\{1,m\}$.
If $i\in\{2,m-1\}$, then $v$ must be adjacent to $v_i$ because $P_m$ is longest.
If $3\le i\le m-2$, then $v$ must be adjacent to $v_i$
because $T_2\not\le T$.

\smallskip
(iii) $\Rightarrow$ (i)\ \
Let $T$ be a hairy path graph containing $P_m$
as a longest induced path subgraph as in Figure~\ref{fig:hairy}.
Suppose each $v_i\in V(P_m)$ for $i=2,\ldots,m-1$
is joined to $k_i$ vertices in $V(T)\setminus V(P_m)$.
Applying the argument in Proposition~\ref{prop:deg1k}
(as in the discussion with the graphs in Figure~\ref{fig:ex1} where $m=4$, $k_2=1$ and $k_3=2$),
we can see that $T\le P_{m+2k}^E$, where $k=k_2+\cdots+k_{m-1}$.
\end{proof}

\section*{Acknowledgements}
The first author was partially supported by NRF-2015R1C1A2A01051589.
The second author was partially supported by NRF-2015R1D1A1A01056723.
This paper was written as part of Konkuk University's research support program for its faculty on sabbatical leave in 2017.

\end{document}